\documentclass[a4paper, 12pt]{article}

\usepackage{amsthm,amssymb, amsmath}

\newtheorem{definition}{Definition}[section]
\newtheorem{proposition}{Proposition}[section]
\newtheorem{corollary}{Corollary}[section]
\newtheorem{remark}{Remark}[section]
\newtheorem{example}{Example}[section]

\def\R{\mathbb{R}\xspace}
\def\Z{\mathbb{Z}\xspace}

\usepackage{amsfonts}
\usepackage{xspace}
\usepackage{stmaryrd}
\usepackage{enumitem}
\usepackage{hyperref}
\usepackage{authblk}
\usepackage{graphicx}

\title{Regions of the type C Catalan arrangement}

\author[$\dagger$]{Anne  Micheli}
\author[$\star$]{Vu Nguyen Dinh}

\affil[$\dagger$]{{\href{mailto:amicheli@irif.fr}{amicheli@irif.fr}
}, IRIF, CNRS and Universit\'e de Paris, Case 7014, 75205
  Paris Cedex 13 France}
\affil[$\star$]{{\href{mailto:ndvu@math.ac.vn}
    {ndvu@math.ac.vn}}, Institute of Mathematics, Vietnam
    Academy of Science and Technology, 18 Hoang Quoc Viet, 10307
    Hanoi, Vietnam}

\begin{document}

 \maketitle
\date{}


\begin{abstract}
  In this paper, we give a bijection between rooted labeled ordered
  forests with a selected subset of their leaves and the regions of
  the type $C$ Catalan arrangement in $\R^n$.  We thus obtain a
  bijective proof of the well-known enumeration formula of these
  regions ${2^n}n! \binom{2n}{n}$.
\end{abstract}


The type $C$ Catalan arrangement $\mathcal{C}_{\left\{-1,0,1
  \right\}}(n)$ in $\R^n$ is the set of the hyperplanes
 $H_{\mathcal{C}_{\left\{-1,0,1
  \right\}}(n)}=\{ x_i - x_j = s, x_i + x_j = s,  2x_i = s, \forall s \in
 \{-1,0,1\}, \forall 1 \leq i <j \leq n\}$.
The regions of $\mathcal{C}_{\left\{-1,0,1 \right\}}(n)$ are the
connected components of
$\R^n\setminus \displaystyle\cup_{H\in \mathcal{C}_{\left\{-1,0,1
    \right\}}(n)}H$.


The number of regions of the Catalan arrangement of type $C$ is well
known and equals ${2^n}n! \binom{2n}{n}$. We can obtain it for
example from Zaslavsky's formula which says that it is enumerated by
$(-1)^n \chi_{H_{\mathcal{C}_{\left\{-1,0,1
  \right\}}(n)} }(-1)$ where $\chi_{H_{\mathcal{C}_{\left\{-1,0,1
  \right\}}(n)} } $ is the
characteristic polynomial of $H_{\mathcal{C}_{\left\{-1,0,1
  \right\}}(n)}$. It is then equals to 
$(-1)^n \chi_{H_{B(n)} }(-1-2n)$ where
$H_{B(n)}=\{ x_i - x_j = 0, x_i + x_j = 0, x_i = 0, \forall 1 \leq i <j \leq n\}$. Then by Theorem 5.5 in
\cite{A96}, the result is deduced.

In this paper, we exhibit a bijection between rooted labeled ordered forests with
a subset of their leaves and these regions and thus retrieve
bijectively their enumeration.

Type $A$ arrangements have been and are still vastly studied in
combinatorics, in particular the problem of bijectively enumerating
the regions of type $A$ arrangements.
The reader can find an introduction to hyperplane arrangement
and its connexions to combinatorics by R.P. Stanley \cite{S07}. The
equation of a hyperplane of  type $A$ arrangement is of the form
$x_i-x_j =s$ with $s$ in $\Z$ and $i, j$ in $\llbracket
1,n\rrbracket = \{1, 2,\ldots, n\}$.  The case of the braid ($s=0$)  
arrangement is easy to understand, the Shi ($s=0,1$) and Catalan ($s=-1,0,1$) cases have nice and
simple formulas which have been bijectively interpreted \cite{AL99,
  PS00, S07, CFV15, DG18}. The number of regions of the Linial ($s=1$)
arrangement was known but it is only recently that O. Bernardi gave  a
bijective interpretation \cite{B18}. His bijection extends to the
regions of many type $A$
arrangements \cite{B18}, including Catalan, Shi and semi-order type
$A$ arrangements.
Our bijections between orders, families of forests and regions of the
type $C$ Catalan arrangement were inspired by the Bernardi bijections.

The results on type $C$ arrangements are less extensive.   In
1996, C.A. Athanasiadis computed the number of regions of the type $C$
Shi arrangement \cite{A96} . The obtained formula is very simple $(2n+1)^n$ and
K. M\'eszar\'os \cite{M13} in 2013 gave a bijective proof for
the number of regions of the type $C$ Shi arrangement,
which was a generalization of the bijection exhibited by
C.A. Athanasiadis and S. Linusson in the type $A$ case \cite{AL99}.
C.A. Athanasiadis also computed among others the number of regions of
the Linial arrangement of type $C$ \cite{A99}. No bijective proof of
this enumeration has yet emerged.

The paper is divided in three sections. In section \ref{sec:regions2orders}, we explain
how to go bijectively from regions of the type $C$ Catalan arrangement to
some orders. Then in section \ref{sec:orders2forests}, we exhibit a bijection between
these orders and rooted labeled ordered forests. Finally, we compute in
section \ref{sec:enumeration}, the number of regions of the type $C$ Catalan
arrangement.

\section{From regions to orders}
\label{sec:regions2orders}
In this section we show that each region $R$ of the type $C$ Catalan
arrangement corresponds bijectively to a specific order between the
variables $x_i$ and $1+x_i$ for any $i$ in $\llbracket -n,n\rrbracket \setminus \left\{0 \right\}$
where $(x_1,\ldots,x_n)$ denotes the coordinates of any point of $R$ and $x_{-i} = -x_i$ for all $i$ in $\llbracket 1,n\rrbracket$.

In the sequel, for any $i$ in $\llbracket -n,n\rrbracket \setminus
\left\{0 \right\}$, we denote by :\\
$\bullet\ \alpha_i^{(0)}$ the variable $x_i$, \qquad $\bullet\ \alpha_i^{(1)}$ the variable $1+x_i$.\\
These notations are derived from the paper of O. Bernardi \cite{B18}.
We also denote by ${\mathcal A}_{2n}$ the alphabet $\{\alpha_i^{(0)}, \alpha_i^{(1)},
\forall i\in \llbracket -n,n\rrbracket \setminus \left\{0 \right\} \}$.

We first define a symmetric annotated 1-sketch and explain its
symmetries. Then, in a second time, we will show that the regions of the
type $C$ Catalan arrangement are in one-to-one correspondence with
symmetric annotated 1-sketches.

\subsection{Symmetric annotated 1-sketch}

\begin{definition}
\label{def:symAnnotSk}
A {\it{symmetric annotated $1$-sketch of size 2n}} is a word $\omega = w_1 ...w_{4n}$ 
that satisfies for all $ i,j \in \llbracket -n,n\rrbracket \setminus \left\{0 \right\}$:
\begin{itemize}
\item[(i)] $ \left\{w_1,...,w_{2n},...,w_{4n} \right\} = {\mathcal A}_{2n},$
\item[(ii)] $ \alpha_{i}^{(0)}$ appears before $\alpha_{i}^{(1)}$,
\item[(iii)] If $ \alpha_{i}^{(0)}$ appears before $\alpha_{j}^{(0)}$ then $ \alpha_{i}^{(1)}$ appears before $\alpha_{j}^{(1)}$,
\item[(iv)] If $ \alpha_{i}^{(0)}$ appears before $\alpha_{j}^{(s)}$ then $ \alpha_{-j}^{(0)}$ appears before $\alpha_{-i}^{(s)}$,
$\forall s \in \{0,1\}$.
\end{itemize}

Let $D^{(1)} (2n)$ be the set of {\it{symmetric annotated $1$-sketches of size 2n}}.
\end{definition}

\begin{example} 
\label{omega}
$\omega = \alpha_{-2}^{(0)}
\alpha_{1}^{(0)}
\alpha_{-2}^{(1)}
\alpha_{3}^{(0)}
\alpha_{-3}^{(0)}
\alpha_{1}^{(1)}
\alpha_{-1}^{(0)}
\alpha_{3}^{(1)}
\alpha_{-3}^{(1)}
\alpha_{2}^{(0)}
\alpha_{-1}^{(1)}
\alpha_{2}^{(1)} \in D^{(1)} (6)$.
\end{example}

\begin{remark}
\label{rk:symAnnotSk} 
  \begin{enumerate}
  \item Condition $(ii)$ of Definition \ref{def:symAnnotSk} implies
    that a symmetric annotated 1-sketch starts with a sequence of
    $\alpha_i^{(0)}$ letters and ends with a sequence of $\alpha_i^{(1)}$
    letters.
  \item \label{rk:symAnnotSk2} Condition $(iv)$ of Definition \ref{def:symAnnotSk} implies
    that the subword of $\omega$ composed of the $\alpha_.^{(0)}$
    letters has the form $\alpha_{i_1}^{(0)}\ldots \alpha_{i_n}^{(0)}
    \alpha_{-i_n}^{(0)} \ldots \alpha_{-i_1}^{(0)}$ with
    $\{|i_1|,\ldots, |i_n|\}=\llbracket 1,n \rrbracket$. Moreover, the subword of
    $\omega$ composed of the $\alpha_.^{(1)}$ letters is exactly $\alpha_{i_1}^{(1)}\ldots \alpha_{i_n}^{(1)}
    \alpha_{-i_n}^{(1)} \ldots \alpha_{-i_1}^{(1)}$.
  \end{enumerate}
\end{remark}

Furthermore, a symmetric annotated 1-sketch is the result of a
specific shuffle between two words on the alphabet ${\mathcal A}_{2n}$ where
one is the symmetric of the other in the following sense :

\begin{definition}
 Let $\omega_1$ be a word on ${\mathcal A}_{2n}$
 that  ends with letter $u$, i.e $\omega_1 = \omega_0 u$. 
We define the symmetric of $\omega_1$ as a word $\overline{\omega}_1 =
\overline{u} \ \overline{\omega}_0$ where $\overline{u} =
\alpha_{-k}^{(1-s)}$ if $u =\alpha_{k}^{(s)}$, $s\in \{0,1\}$
and $\overline{\omega}_0$ is recursively defined in the same way.
\end{definition}

\begin{example}
\label{ex:symsk}
The symmetric of $\omega_1 =\alpha_{-2}^{(0)}
\alpha_{1}^{(0)}
\alpha_{-2}^{(1)}
\alpha_{3}^{(0)}
\alpha_{1}^{(1)}
\alpha_{3}^{(1)} $ is $\overline{\omega}_1 = \alpha_{-3}^{(0)}
\alpha_{-1}^{(0)}
\alpha_{-3}^{(1)}
\alpha_{2}^{(0)}
\alpha_{-1}^{(1)}
\alpha_{2}^{(1)}$.
\end{example}

Now, a symmetric annotated 1-sketch $\omega$ is the combination of two
symmetric words $\omega_1$ and $\omega_2 = \overline{\omega}_1$. As a matter of fact, 
we will now explain how we obtain $\omega_1$ and $\omega_2$ from $\omega$. 
We call words of the form $\omega_1$, annotated 1-sketches which formal
definition is:
\begin{definition}
\label{definition: annosketch}
  An annotated 1-sketch of size $n$ is defined by $2n$ letters
  $\alpha_{j_k}^{(0)}$ and $\alpha_{j_k}^{(1)}$, $k$ in $\llbracket
  1,n\rrbracket$ such that $\{| j_1|,\ldots ,| j_n|\}= \llbracket
  1,n\rrbracket$ and which satisfies conditions
  $(ii)$ and $(iii)$ of Definition \ref{def:symAnnotSk}. 

  We denote  by $A_{n,s}$,
  $n \leq s \leq 2n-1$, the set of annotated 1-sketches where the
  rightmost letter $\alpha_{.}^{(0)}$ 
  is at position $s$.




\end{definition}


Thus we get that :
\begin{proposition}
\label{prop:sketchdecomp}
  Any symmetric annotated 1-sketch $\omega$ is the composition of
  an annotated 1-sketch $\omega_1$ and its symmetric $\overline{\omega}_1$.
\end{proposition}

\begin{proof}
  We define $\omega_1$ as the subword of $\omega$ composed of the $n$
  leftmost $ \alpha_{.}^{(0)}$ letters and the corresponding
  $\alpha_{.}^{(1)}$ letters (if $ \alpha_{i}^{(0)}$ appears in
  $\omega_1$ then $ \alpha_{i}^{(1)}$ appears in $\omega_1$). 
  Remark \ref{rk:symAnnotSk}(\ref{rk:symAnnotSk2}) implies that
  $\alpha_{i}^{(0)}$ and $\alpha_{-i}^{(0)}$ cannot both belong to the
  set of the $n$ leftmost $ \alpha_{.}^{(0)}$ letters of
  $\omega$. Thus, it  is easy to see that $\omega_1$ is an annotated 1-sketch.

  This remark and condition $(iv)$ of Definition \ref{def:symAnnotSk} also imply that
  $\omega_2$ the subword of $\omega$ composed of the letters not in
  $\omega_1$ is the symmetric of $\omega_1$.  
\end{proof}

\begin{example} 
$\omega$ of Example \ref{omega} is composed of $\omega_1 =\alpha_{-2}^{(0)}
\alpha_{1}^{(0)}
\alpha_{-2}^{(1)}
\alpha_{3}^{(0)}
\alpha_{1}^{(1)}
\alpha_{3}^{(1)}\in A_{3,4}$ and $\overline{\omega}_1$.
\end{example}

Conversely, for any annotated 1-sketch $\omega_1$, we can construct a set
of  symmetric annotated 1-sketches, the result of shuffles
between $\omega_1$ and $\overline{\omega}_1$. We first give the definition of these
shuffles and then prove the assertion.

\begin{definition}
\label{def:shuffles}
  Let $\psi = \alpha_{j_1}^{(1)}\ldots \alpha_{j_k}^{(1)}$. We define
  the set of shuffles $\psi\bowtie \overline{\psi}$ recursively with
  $\psi\bowtie \overline{\psi} = \{\epsilon\}$ if $\psi$ is the empty word $\epsilon$, as the
  set of following words:
  \begin{itemize}
  \item $\alpha_{-j_k}^{(0)} \psi'\bowtie \overline{\psi'}
    \alpha_{j_k}^{(1)}$ with $\psi' = \alpha_{j_1}^{(1)}\ldots
    \alpha_{j_{k-1}}^{(1)}$ ($\psi'=\epsilon$ if $k=1$),
  \item $\alpha_{j_1}^{(1)}\ldots \alpha_{j_i}^{(1)}\alpha_{-j_k}^{(0)} \psi'\bowtie \overline{\psi'}
    \alpha_{j_k}^{(1)}\alpha_{-j_i}^{(0)}\ldots \alpha_{-j_1}^{(0)}$ with $\psi' = \alpha_{j_{i+1}}^{(1)}\ldots
    \alpha_{j_{k-1}}^{(1)}$ ($\psi'=\epsilon$ if $i=k-1$), $\forall 1\leq i \leq k-1$,  
  \item $\alpha_{j_1}^{(1)}\ldots \alpha_{j_k}^{(1)}\alpha_{-j_k}^{(0)} \ldots \alpha_{-j_1}^{(0)}$. 
  \end{itemize}
\end{definition}

\begin{example} The set of shuffles $\psi\bowtie \overline{\psi}$ with
  $\psi = \alpha_{j_1}^{(1)} \alpha_{j_2}^{(1)}$ is composed of the
  four words $\alpha_{-j_2}^{(0)} \alpha_{-j_1}^{(0)}\alpha_{j_1}^{(1)}
\alpha_{j_2}^{(1)}$, $\alpha_{-j_2}^{(0)} \alpha_{j_1}^{(1)}\alpha_{-j_1}^{(0)}
\alpha_{j_2}^{(1)}$, $\alpha_{j_1}^{(1)} \alpha_{-j_2}^{(0)}\alpha_{j_2}^{(1)}
\alpha_{-j_1}^{(0)}$ and $\alpha_{j_1}^{(1)} \alpha_{j_2}^{(1)}\alpha_{-j_2}^{(0)}
\alpha_{-j_1}^{(0)}$.
 
\end{example}

\begin{definition}
\label{def:sketchesshuffle}
Let   $\omega_1= \omega_0 \alpha_{j_n}^{(0)} \psi$ with $\psi=\alpha_{j_{s-n+1}}^{(1)}
\alpha_{j_{s-n+2}}^{(1)}... \alpha_{j_{n-1}}^{(1)} \alpha_{j_n}^{(1)}$,
be an annotated 1-sketch. Then 
$\omega_1\bowtie \overline{\omega}_1 = \omega_0 \alpha_{j_n}^{(0)} \left[\psi
\bowtie \overline{\psi} \right ] \alpha_{-j_n}^{(1)} 
\overline{\omega}_0 = \{\omega_0 \alpha_{j_n}^{(0)} u \alpha_{-j_n}^{(1)} 
\overline{\omega}_0, u \in \psi
\bowtie \overline{\psi}\}$.


\end{definition}



\begin{proposition}
\label{proposition: shufflesym}
  For any annotated 1-sketch $\omega_1$ of size $n$, 
  $\omega_1\bowtie \overline{\omega}_1 \subset D^{(1)}(2n)$. 
\end{proposition}

\begin{example}
 $\omega_1 =\alpha_{-2}^{(0)}
\alpha_{1}^{(0)}
\alpha_{-2}^{(1)}
\alpha_{3}^{(0)}
\alpha_{1}^{(1)}
\alpha_{3}^{(1)} \in A_{3,4}$.
Then $\omega_1\bowtie \overline{\omega}_1$ is the set of $4$ elements:\\
$\alpha_{-2}^{(0)}
\alpha_{1}^{(0)}
\alpha_{-2}^{(1)}
\alpha_{3}^{(0)}
\alpha_{-3}^{(0)}
\alpha_{1}^{(1)}
\alpha_{-1}^{(0)}
\alpha_{3}^{(1)}
\alpha_{-3}^{(1)}
\alpha_{2}^{(0)}
\alpha_{-1}^{(1)}
\alpha_{2}^{(1)}$,
$\alpha_{-2}^{(0)}
\alpha_{1}^{(0)}
\alpha_{-2}^{(1)}
\alpha_{3}^{(0)}
\alpha_{-3}^{(0)}
\alpha_{-1}^{(0)}
\alpha_{1}^{(1)}
\alpha_{3}^{(1)}
\alpha_{-3}^{(1)}
\alpha_{2}^{(0)}
\alpha_{-1}^{(1)}
\alpha_{2}^{(1)}$,
$\alpha_{-2}^{(0)}
\alpha_{1}^{(0)}
\alpha_{-2}^{(1)}
\alpha_{3}^{(0)}
\alpha_{1}^{(1)}
\alpha_{-3}^{(0)}
\alpha_{3}^{(1)}
\alpha_{-1}^{(0)}
\alpha_{-3}^{(1)}
\alpha_{2}^{(0)}
\alpha_{-1}^{(1)}
\alpha_{2}^{(1)}$,
$\alpha_{-2}^{(0)}
\alpha_{1}^{(0)}
\alpha_{-2}^{(1)}
\alpha_{3}^{(0)}
\alpha_{1}^{(1)}
\alpha_{3}^{(1)}
\alpha_{-3}^{(0)}
\alpha_{-1}^{(0)}
\alpha_{-3}^{(1)}
\alpha_{2}^{(0)}
\alpha_{-1}^{(1)}
\alpha_{2}^{(1)}.$ 
\end{example}

\begin{proof}
We must prove that any word of $\omega_1\bowtie \overline{\omega}_1$ is a
symmetric annotated 1-sketch, meaning that it verifies conditions
$(i)$ to $(iv)$ of Definition \ref{def:symAnnotSk}.

Conditions $(i)$, $(ii)$ and $(iii)$ are straightforward since $\omega_1$ and
$\overline{\omega}_1$ are annotated 1-sketches, each one the symmetric of
the other, and their letters are not permuted.

Let $\omega_1= \omega_0 \alpha_{j_n}^{(0)} \psi$ with $\psi=\alpha_{j_{s-n+1}}^{(1)}
\alpha_{j_{s-n+2}}^{(1)}... \alpha_{j_{n-1}}^{(1)}\alpha_{j_n}^{(1)}$. A
word of $\omega_1 \bowtie \overline{\omega}_1$ is either
$\omega_1\overline{\omega}_1$ which obviously verifies condition $(iv)$, or has one of the
following form and we can thus check recursively that it verifies condition $(iv)$ :
\begin{itemize}
\item $\omega_0 \alpha_{j_n}^{(0)} \alpha_{-j_n}^{(0)} \left [\psi' \bowtie
    \overline{\psi'}\right ] \alpha_{j_n}^{(1)} \alpha_{-j_n}^{(1)}
  \overline{\omega}_0$, $\psi'=\alpha_{j_{s-n+1}}^{(1)}
\alpha_{j_{s-n+2}}^{(1)}... \alpha_{j_{n-1}}^{(1)}$, and thus
 $\alpha_{-j_n}^{(0)}$ appears before
  $\alpha_{j_t}^{(1)}$ and $\alpha_{-j_t}^{(0)}$ appears before
  $\alpha_{j_n}^{(1)}$, for any $t$ in $\llbracket s-n+1,
  n-1\rrbracket$,
\item  $\omega_0
  \alpha_{j_n}^{(0)}\alpha_{j_{s-n+1}}^{(1)}... \alpha_{j_{k}}^{(1)}
  \alpha_{-j_n}^{(0)} [ \psi' \bowtie \overline{\psi'} ]
  \alpha_{j_{n}}^{(1)}
  \alpha_{-j_{k}}^{(0)}...\alpha_{-j_{s-n+1}}^{(0)}
  \alpha_{-j_{n}}^{(1)} \overline{\omega}_0$,
  $\psi'=\alpha_{j_{k+1}}^{(1)}... \alpha_{j_{n-1}}^{(1)}$, and thus $\alpha_{-j_n}^{(0)}$ appears before
  $\alpha_{j_t}^{(1)}$ and $\alpha_{-j_t}^{(0)}$ appears before
  $\alpha_{j_n}^{(1)}$, for any $t$ in $\llbracket k+1,
  n-1\rrbracket$.
\end{itemize}
\end{proof}

\subsection{Bijection between regions and symmetric annotated 1-sketches}
A symmetric annotated 1-sketch corresponds to a specific order between the variables $x_i$ and $1+x_i$ for any $i$ in
$\llbracket -n,n\rrbracket \setminus \left\{0 \right\}$. We show here that these orders are bijectively related to the coordinates of the points of the regions of the type $C$ Catalan arrangement.
\begin{proposition}
\label{prop:bijArr2sk}
There is a one to one correspondence between regions of the type $C$
Catalan arrangement in $\R^n$ and the symmetric annotated 1-sketches of size $2n$.
\end{proposition}

\begin{proof}
Observe that for all $x \in {\mathbb{R}}^n$, if there exist $i,j \in
\llbracket -n,n\rrbracket \setminus \left\{0 \right\}$ and $s,t \in
\left\{0,1 \right\}$ such that  $x_i +s =x_j +t$ then $x \in {\cup}_{
  H \in \mathcal{C}_{\left\{-1,0,1 \right\}} (n)} H$. Therefore,  for
any $x = \left\{ x_1,..., x_n \right\}$ that belongs to $ {\mathbb{R}}^n
\setminus {\cup}_{ H \in \mathcal{C}_{\left\{-1,0,1 \right\}} (n)} H$,
the elements of $ \left\{ x_i +s : i \in \llbracket -n,n\rrbracket \setminus
  \left\{0 \right\}, s \in \{0,1\} \right\}$ are all distinct, with
$x_{-i}= -x_i$ for all $i$. We define $ {\sigma} (x) =w_1
w_2...w_{4n}$, where $w_p = \alpha_{i}^{(s)}$ if $z_p =
x_i +s$ with $\left\{z_1 < z_2 <...< z_{4n} \right\} = \left\{ x_i +s
  : i \in \llbracket -n,n\rrbracket \setminus \left\{0 \right\}, s \in
  \{0,1\} \right\}$. 

${\sigma} (x)$ obviously satisfies conditions $(i)-(iii)$ of Definition
\ref{def:symAnnotSk}. 
We now prove that ${\sigma} (x)$
satisfies condition $(iv)$ of Definition \ref{def:symAnnotSk}. Indeed,
if $\alpha_{i}^0$ appears before $\alpha_{j}^s$ with
$s \in \left\{0,1 \right\}$ then $x_i < x_j +s$, hence
$x_{-j} < x_{-i} +s$. It induces that $\alpha_{-j}^0$ appears before
$\alpha_{-i}^s$. Therefore ${\sigma} (x)$ is a symmetric annotated
1-sketch of size $2n$. The mapping ${\sigma}$ is constant over each
region of $\mathcal{C}_{\left\{-1,0,1 \right\}} (n)$. Thus,
${\sigma} $ is a mapping from the regions of
$\mathcal{C}_{\left\{-1, 0, 1 \right\}}(n)$ to ${D}^{(1)} (2n)$.

The mapping $\sigma$ satisfies, $ x_i - x_j < s$ if $\alpha_{i}^{(0)}$ appears before $\alpha_{j}^{(s)}$  and $x_i - x_j > s$ otherwise,  for all $i, j \in \llbracket -n,n\rrbracket \setminus \left\{0 \right\}$
and all $s \in \left\{0,1 \right\}$. Thus, $\sigma$ is injective. Finally, for any symmetric annotated 1-sketch $ \omega =w_1 w_2...w_{4n}$, there exists $ x \in  {\mathbb{R}}^n \setminus {\cup}_{ H \in \mathcal{C}_{\left\{-1,0,1
 \right\}} (n)} H$ such that $ \sigma(x) = \omega$. Indeed, we define $x \in {\sigma}^{-1} (\omega)$  and
$z_1, ...,  z_{4n}$ by applying the following rule for $p = 1, 2,..., 4n$: if
$w_ p =\alpha_{i}^{(0)}$ then 
$z_p = z_{p-1} + 1/(2n+1)$ and 
$ z_p = x_i$, while if $
w_ p =\alpha_{i}^{(1)}$ then $z_p = x_i +1$. Therefore ${\sigma}
$ is a bijection.
\end{proof}


\section{From orders to forests}
\label{sec:orders2forests}
In this section, we present a bijection between the symmetric
annotated 1-sketches and some rooted labeled ordered forests that we call symmetric
forests. We will first define these forests and then expose
the bijection.

\subsection{Symmetric forests}

In order to define a  symmetric forest, we need to
introduce the notion of sub-descendant in a forest. 

For any rooted labeled ordered forest $F$, we say that we read the
nodes of $F$ in BFS order if we list the labels of the nodes of $F$ in a
breadth-first search starting from the root.

\begin{definition}
  Let $i$ and $j$ be two nodes in a rooted ordered forest. We say that
  $i$ is a sub-descendant of $j$ if $i$ appears after $j$ and strictly
  before any child of $j$ in the BFS order. We
  also say that $i$ and $j$ satisfy the sub-descendant property (SDP)
  if $i$ is a sub-descendant of $j$ implies that $-j$  is a sub-descendant of $-i$.
\end{definition}

\begin{definition}
\label{definition:symrootedforest}
 A symmetric forest with $2n$ nodes is a rooted labeled ordered forest that satisfies:
 \begin{enumerate}[label=(\roman*)]
 \item the first $n$ nodes read in BFS order are labeled $e_1,
   ..., e_n$ such that $\{| e_1|,\ldots,|e_n|\} =\llbracket
   1,n\rrbracket$,
 \item the last $n$ nodes read in BFS order are labeled $e_{n+1},
   ..., e_{2n}$ such that $e_{n+j} = {-e}_{n-j+1}$ with $j \in \llbracket
   1,n\rrbracket$,
\item for every two nodes $i,j$, $i$ and $j$ satisfy the sub-descendant property.
 \end{enumerate}
We denote by $F_{S}(2n)$ the symmetric forests with $2n$ nodes.
\end{definition}

\begin{example}
    For the symmetric forest $G$ in Figure \ref{fig:symForest}, $1$ is a sub-descendant of $-2$ and $2$ is a sub-descendant of $-1$, hence $\left\{1,-2 \right\}$ satisfy the sub-descendant property. Moreover, $G \in F_{S}(6)$.

 \end{example}

As a matter of fact, a symmetric forest is composed
of two sub-forests where one is the symmetric of the other in the following sense :

\begin{definition}
\label{definition:symofrootedforest}
  Let $F$ be a rooted ordered forest defined on $n$ labeled nodes
  $e_1,.., e_n$. We define the symmetric of $F$ as a rooted ordered
  forest $\overline{F}$ with n labeled nodes $-e_n,..., -e_1$ such
  that for all $i \neq j \in  \llbracket 1,n\rrbracket$, $-e_i$ is a
  sub-descendant of $-e_j$ in $\overline{F}$ if and only if $e_j$ is a
  sub-descendant of $e_i$ in $F$.
\end{definition}

We now explain how to decompose a symmetric forest $G$
into a
forest $F$ and its symmetric. $F$ is the sub-forest of $G$ defined on
the first $n$ nodes read in BFS order. 

     \begin{figure}[h]
       \centering
       \includegraphics[width=300.600px] {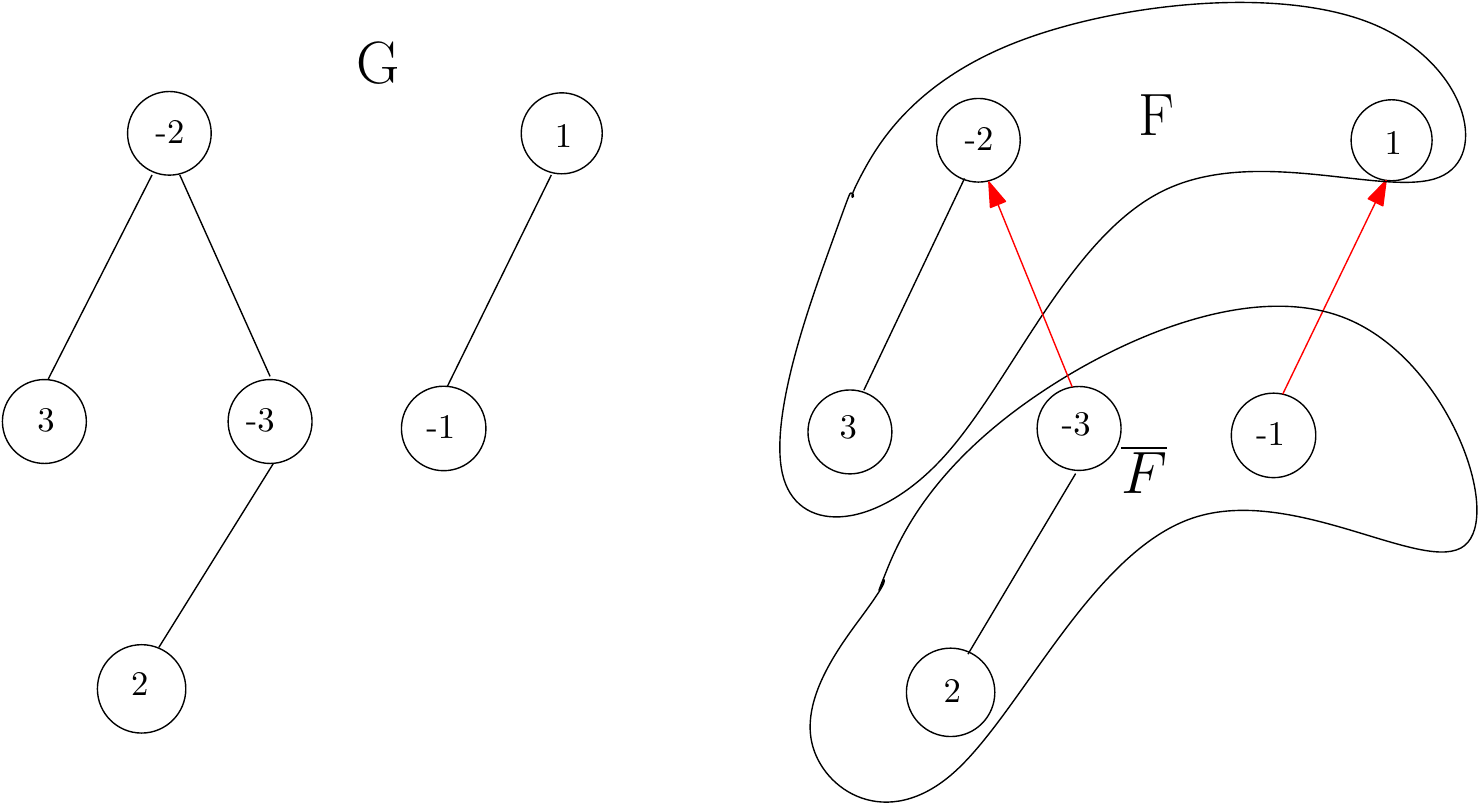}
       \caption{a symmetric forest $G \in F_{S}(6)$,
         result of a shuffle between $F$ and $\overline{F}$.}
       \label{fig:symForest}
     \end{figure}

Thus we have that :
\begin{proposition}
\label{prop:forestdecomp}
  A symmetric forest with $2n$ nodes is the composition of a rooted labeled
  ordered forest with $n$ nodes and its symmetric.
\end{proposition}

\begin{proof} 
$F$ is the sub-forest of $G$ defined by the first $n$
nodes read in BFS order. Then the sub-forest of $G$ corresponding to
the $n$ last nodes read in BFS order is the forest, symmetric of $F$,
by Definition \ref{definition:symrootedforest} 
and Definition \ref{definition:symofrootedforest}.
\end{proof}

Conversely, any shuffle between any rooted labeled ordered forest $F$ and
its symmetric, is in bijection with a symmetric forest. We
first give the definition of a special leaf, then the definition of the shuffles
between a forest and its symmetric (see Figure
\ref{fig:symOrderedForest}) and finally we prove the
assertion.

\begin{definition}
 In a rooted ordered forest with $n$ labeled nodes $e_1, ..., e_n$
 such that $\{| e_1|,\ldots,|e_n|\}=\llbracket 1,n\rrbracket$, the special leaves are the leaves which are
 after the last internal node in the BFS order. If a forest $F$ has
 only leaves, we consider that its last internal node is a fictif node,
 parent of the leaves of $F$. Let us call $F_{n,s},
 1 \leq s \leq n$,  the set of rooted labeled ordered forests of size
 $n$ with $s$ special leaves. 
\end{definition}

\begin{example}
  The rooted labeled ordered forest $F$ of Figure \ref{fig:symForest}, has two
  special leaves, $1$ and $3$.
\end{example}

\begin{definition}
\label{definition:shufflesrootedforest}
Let $F$ be a rooted ordered forest defined on $n$ labeled nodes
$e_1,.., e_n$, ordered in BFS order and such that
$\{|e_1|,\ldots,|e_n|\}=\llbracket 1,n\rrbracket$, with $s$ special leaves. The set of
shuffles between $F$ and its symmetric $\overline{F}$, $F \bowtie
\overline{F}$, is the set of
forests obtained when we connect $s$ edges from
$\left\{-e_n, -e_{n-1},...,-e_{n-s+ 1} \right\}$ to
$\left\{e_{n-s}, e_{n-s+1},...,e_{n-1}, e_n \right\}$ such that any
pair $(u,v)$, $u$ in $\left\{-e_n, -e_{n-1},...,-e_{n-s+ 1} \right\}$ and
$v$ in $\left\{e_{n-s},...,e_{n-1}, e_n \right\}$, satisfies the sub-descendant
property and the sequence of the
nodes read in BFS order is $e_1,\ldots, e_n, -e_n,
-e_{n-1},...,-e_1$. We say that $\left\{-e_n, -e_{n-1},...,-e_{n-s+ 1} \right\}$ and
$\left\{e_{n-s},...,e_{n-1}, e_n \right\}$ satisfy the sub-descendant
property.
\end{definition}

\begin{proposition}
\label{proposition: shuffleforests}
  For any rooted ordered forest $F$ with $n$ nodes labeled with
  $e_1,\ldots, e_n$ such that $\{| e_1|,\ldots,|e_n|\}=\llbracket
  1,n\rrbracket$, the set $F \bowtie \overline{F}$ is
  a set of symmetric forests with $2n$ nodes.
\end{proposition}
\begin{proof}
Conditions $(i)$ and $(ii)$ of Definition
\ref{definition:symrootedforest} are verified by definition of the shuffle.

Notice that for any connection of $s$ edges from
$\left\{-e_n, -e_{n-1},...,-e_{n-s+ 1} \right\}$ to
$\left\{e_{n-s}, e_{n-s+1},\right.$
$\left. ...,e_{n-1}, e_n \right\}$,
$\left\{-e_{n-s}, -e_{n-s-1},...,-e_1 \right\}$ always satisfies the
sub-descendant property with $\left\{e_1, e_2,...,e_n \right\}$, and
$\left\{-e_{n}, -e_{n-1},...,-e_{n-s+1} \right\}$ always satisfies the
sub-descendant property with $\left\{e_1, e_2,...,e_{n-s-1} \right\}$.
By Definition \ref{definition:shufflesrootedforest}
,$\left\{-e_n, -e_{n-1},...,-e_{n-s+ 1} \right\}$ and
$\left\{e_{n-s},...,e_{n-1}, e_n \right\}$ satisfy the sub-descendant
property. Therefore, for any forest $G$ in
$F \bowtie \overline{F}$, for every two nodes $e_i,e_j$, $e_i$ and
$e_j$ satisfy the sub-descendant property. Thus, $G$ is a symmetric
forest.
\end{proof}

\begin{figure}[h]
  \centering
  \includegraphics[width=450.810px] {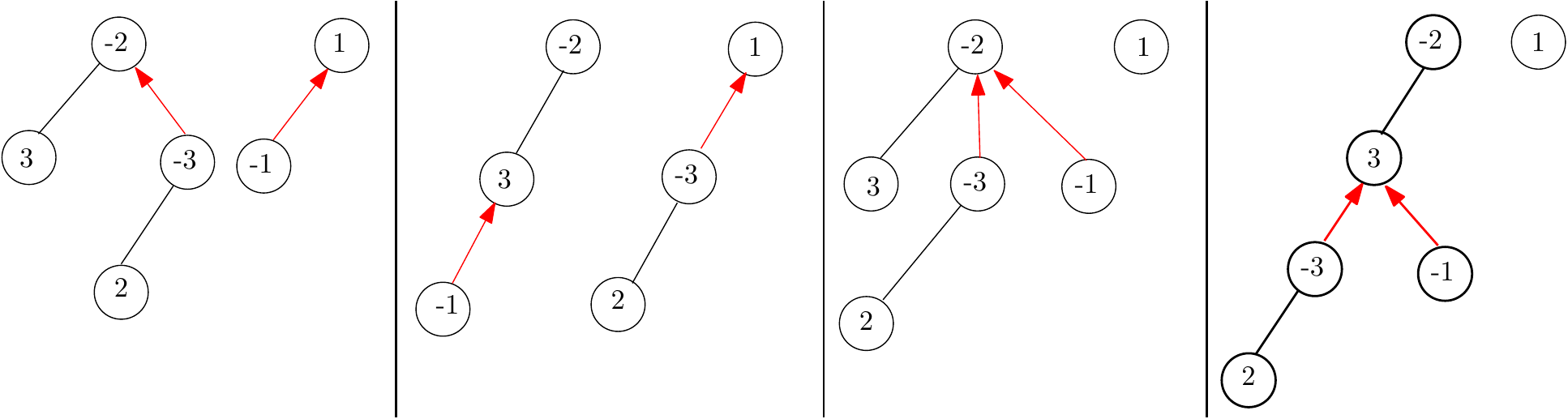}
  \caption{the set of shuffles between the forests F and $\overline{F}$ of Figure \ref{fig:symForest}.}
  \label{fig:symOrderedForest}
\end{figure}

\vspace{0.5cm}

\subsection{Bijection between symmetric annotated 1-sketches and
  symmetric forests}
We will show here that a symmetric annotated 1-sketch corresponds
bijectively to a symmetric forest. Moreover, the
decomposition of a symmetric annotated 1-sketch (see Proposition \ref{prop:sketchdecomp}) 
corresponds to the decomposition of a symmetric 
forest (see Proposition \ref{prop:forestdecomp}) .

\begin{proposition}
\label{proposition:bijsym2n}
  There is a one to one correspondence between symmetric annotated
  1-sketches of size $2n$ and symmetric forests of size
  $2n$.
\end{proposition}

\begin{proof}
   We now prove the proposition in 3 steps:
   \paragraph{Step 1:} we first present an algorithm to get the
   symmetric forest from a symmetric annotated
   1-sketch $\omega$ of $D^1{(2n)}$. We define the map $\phi$ between
   $D^1{(2n)}$ and $F_{S}{(2n)}$ by the following algorithm (see Figure \ref{fig:algorsymOrderedForest}):
   \begin{itemize}
   \item[(i)] Read $\omega$ from left to right.
   \item[(ii)] When $\alpha_{i}^{(0)}$ is read, create a node $i$ such
     that, if $\alpha_{i}^{(0)}$ is not the first letter, if the preceding letter is $\alpha_{j}^{(0)}$ then $i$
     becomes the next right sibling of $j$, and if the preceding letter is
     $\alpha_{j}^{(1)}$ then $i$ becomes the leftmost child of $j$.
   \end{itemize}

\begin{figure}[h]
  \centering
  \includegraphics[width=400.650px] {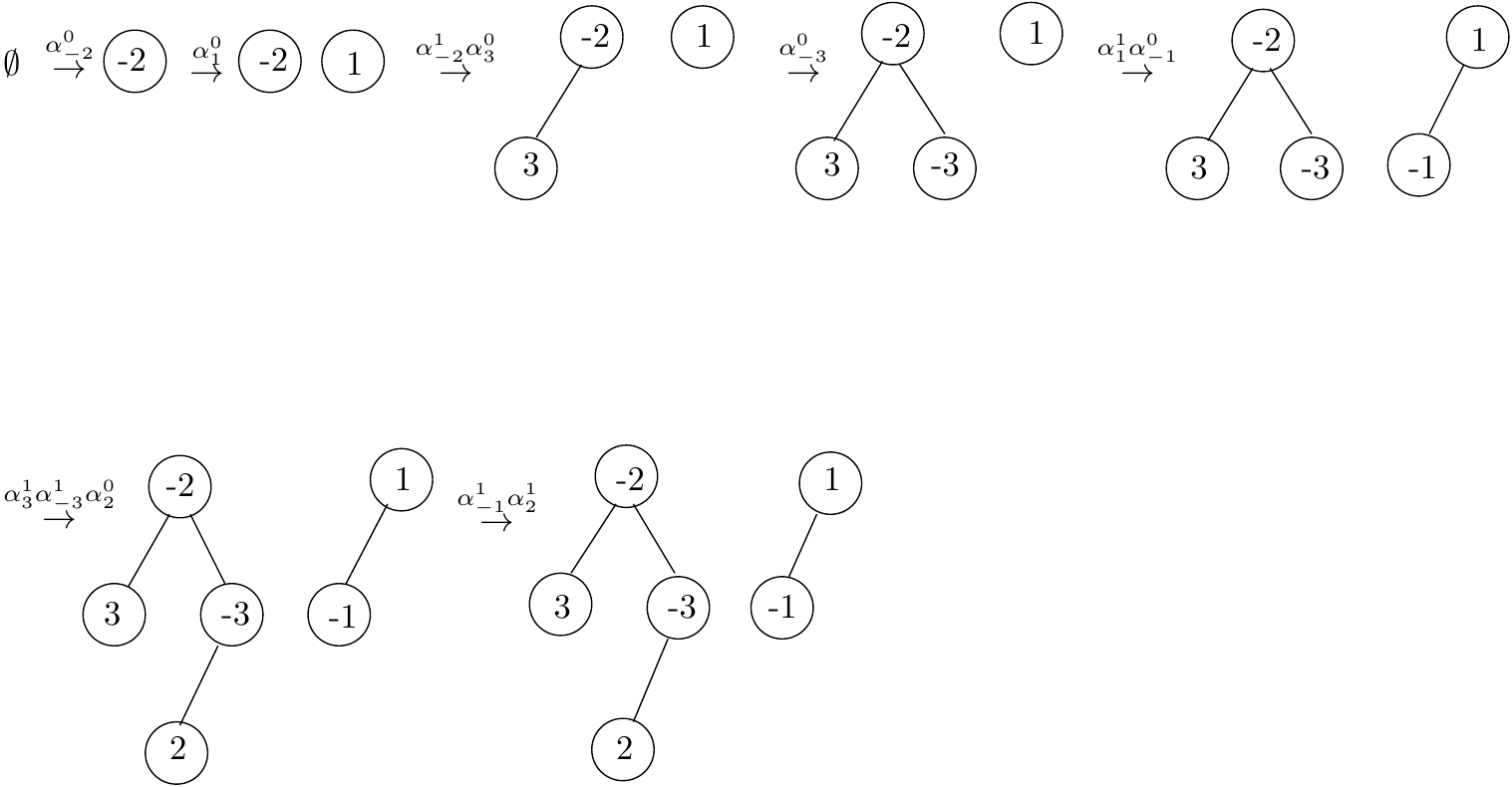}
  \caption{algorithm to get the symmetric forest $\phi\left(\scriptstyle
\alpha_{-2}^{(0)}
\alpha_{1}^{(0)}
\alpha_{-2}^{(1)}
\alpha_{3}^{(0)}
\alpha_{-3}^{(0)}
\alpha_{1}^{(1)}
\alpha_{-1}^{(0)}
\alpha_{3}^{(1)}
\alpha_{-3}^{(1)}
\alpha_{2}^{(0)}
\alpha_{-1}^{(1)}
\alpha_{2}^{(1)}\normalsize\right)$}
  \label{fig:algorsymOrderedForest}
\end{figure}

First note that $\alpha_{i}^{(0)}$ and $\alpha_{-i}^{(0)}$ cannot be both in the first $n$ $ \alpha_{.}^{(0)}$ letters. By definition, the forest  $\phi(\omega)$ has $n$ first nodes labeled by the first $n$ $ \alpha_{.}^{(0)}$-letters. And the last $n$ nodes are defined symmetrically as in the symmetric annotated 1-sketches $D^1{(2n)}$.

Second, remark that:
\begin{remark}
\label{rem:bijsk2forests}
\begin{enumerate}
\item  If $ \alpha_{i}^{(1)}$ is not followed by an $
  \alpha_{.}^{(0)}$-letter then node $i$ is a leaf.
\item If $\alpha_{i}^{(0)}$ appears before $ \alpha_{j}^{(0)}$ in
  $\omega$ then node $i$ appears before node $j$ in the BFS
  order of the nodes of the obtained forest.
\item The property  ``$ \alpha_{i}^{(0)}$ appears before $
  \alpha_{j}^{(0)}$ then $ \alpha_{-j}^{(0)}$ appears before $
  \alpha_{-i}^{(0)}$'', implies that ``$ i$ appears before $j$
  then $-j$ appears before $-i$ in the BFS order of the nodes of the
  obtained forest''.
\item \label{rem:SDP}The property ``$\alpha_{i}^{(0)}$ appears before
  $\alpha_{j}^{(1)}$ then $\alpha_{-j}^{(0)}$ appears before
  $\alpha_{-i}^{(1)}$'' is equivalent to
  ``$\alpha_{j}^{(0)}...\alpha_{i}^{(0)}...\alpha_{j}^{(1)}$ then
  $\alpha_{-i}^{(0)}...\alpha_{-j}^{(0)}...\alpha_{-i}^{(1)}$''. So,
  if $i$ is a sub-descendant of $j$ then $-j$ is a sub-descendant of
  $-i$.
\end{enumerate}
\end{remark}
From these last two remarks, it is clear that $\phi(\omega)$ is a symmetric forest. 

 \paragraph{Step 2:} before showing that  $\phi$ is a bijection, we
 describe the inverse mapping $\psi$. Let $G \in F_{S}(2n)$ and $e_1,
 e_2, ..., e_{2n}$ be the $2n$ nodes in $G$ read in BFS order. Let $\psi(G)$ be the word $\omega$ defined inductively as follow:
 \begin{itemize}

\item Read the vertices in BFS order. $\omega_1 = \alpha_{e_1}^{(0)}$.


\item For any $2 \leq j \leq 2n$, if $e_j$ is the next right sibling of $e_{j-1}$ then $\omega_j = \omega_{j-1}  \alpha_{e_j} ^{(0)}$, if $e_j$ is the leftmost child of
 $e_{i}$ then $\omega_j = \omega_{j-1}\alpha_{e_i} ^{(1)}\alpha_{e_j} ^{(0)}$.

\item $\omega = \omega_{2n}\alpha_{e_{2n-s+1}} ^{(1)}\alpha_{e_{2n-s+2}}^{(1)
  }... \alpha_{e_{2n}}^{(1)} = \omega_{2n}\alpha_{-e_s}
  ^{(1)}\alpha_{-e_{s-1}} ^{(1) }... \alpha_{-e_1}^{(1)}$ if $G$ has
  $s$ special leaves.
\end{itemize}

 Note that for all $G \in F_{S}(2n)$, the word $\psi(G)$ satisfies
 the properties (i)-(iv) of symmetric annotated 1-sketches. Hence
 $\psi$ is a mapping from $F_{S}(2n)$ to 
 $D^1{(2n)}$.

\paragraph{Step 3:} it is easy to prove that $\psi(\phi(D^1{(2n)}))=
D^1{(2n)}$ and $\phi(\psi(F_{S}(2n)))= F_{S}(2n)$.

\end{proof}

From Propositions \ref{prop:bijArr2sk} and \ref{proposition:bijsym2n}, we get that:
\begin{corollary}
\label{cor:regions2symforests}
  $\Phi = \phi \circ\sigma$ is a bijection from the regions of the Catalan arrangement $\mathcal{C}_{\left\{-1,0,1 \right\}}(n)$ to the symmetric forests $F_S (n)$.
\end{corollary}
  


\vspace{0.3cm}

The bijection $\phi$ induces a bijection between the annotated
1-sketches of size $n$ with rightest $\alpha_.^{(0)}$-letter at
position $s$ and the rooted labeled ordered forests with $2n-s$ special leaves.

\begin{proposition}
\label{prop:symbij}
 The mapping $\phi$ induces a bijection between $A_{n,s}$ and $F_{n,2n-s},n \leq s \leq 2n-1$.
\end{proposition}
\begin{proof}
 Let $\omega_1 \in A_{n,s}$. It means that $\omega_1= \omega_0
 \alpha_{j_n}^{(0)} \alpha_{j_{s-n+1}}^{(1)}
 \alpha_{j_{s-n+2}}^{(1)}... \alpha_{j_{n-1}}^{(1)}
 \alpha_{j_n}^{(1)}$. 

From the first step of the proof of Proposition
\ref{proposition:bijsym2n}, we have that:
\begin{itemize}
\item $\alpha_{j_1}^{(0)}\ldots \alpha_{j_n}^{(0)}$ represent the
  nodes $j_1,\ldots , j_n$ read in BFS order in $\phi(\omega_1)$,
\item If $ \alpha_{i}^{(1)}$ is not followed by an $
  \alpha_{.}^{(0)}$-letter then node $i$ is a leaf.
\item The last $ \alpha_{.}^{(1)}$-letter followed by a
  $\alpha_{.}^{(0)}$-letter is $\alpha_{j_{s-n}}^{(1)}$. This implies
  that the last internal node in the BFS order of the nodes of $\phi(\omega_1)$ is $j_{s-n}$.
\end{itemize}
 Thus, $\phi(\omega_1)$ is
 a rooted ordered forest with $n$ labeled nodes $j_1, ..., j_n$ where
 the nodes  $j_{s-n+1}, j_{s-n+2}, ..., j_{n} $ are $2n-s$ special
 leaves and $\phi(\omega_1) \in F_{n,2n-s}$. Conversely, let $ F \in
 F_{n,2n-s}$ with $2n-s$ special leaves $j_{s-n+1}, j_{s-n+2}, ...,
 j_{n} $. Then $\alpha_ {j_{n}} ^{(0)}$ is at the $s^{th}$ position in $\psi(F)$,  hence it belongs to 
$A_{n,s}$. 

It is easy to prove that $\psi(\phi(A_{n,s}))= A_{n,s}$. Similarly, $\phi(\psi(F_{n,2n-s}))= F_{n,2n-s} $.
\end{proof}

We now show that the the different possible
shuffles between an annotated 1-sketch and its symmetric correspond by
$\phi$ to the different possible
shuffles between a rooted labeled ordered forest and its symmetric.

\begin{proposition}
\label{prop:compatible}
The bijections $\phi$ and $\psi$ are compatible with shuffles and
symmetrics. Indeed, let $\omega_1$ be annotated 1-sketch, then
$\phi( \overline{\omega}_1)= \overline{\phi(\omega_1)}$ and $\phi(
\omega_1\bowtie\overline{\omega}_1)=\phi(\omega_1) \bowtie
\overline{\phi(\omega_1)}$.
\end{proposition}

\begin{proof}
 Let $\omega_1= \omega_0 \alpha_{j_n}^{(0)} \alpha_{j_{s-n+1}}^{(1)}
 \alpha_{j_{s-n+2}}^{(1)} \ldots \alpha_{j_{n-1}}^{(1)}
 \alpha_{j_n}^{(1)} \in A_{n,s}$.

From Proposition \ref{prop:symbij}, we get $\phi(
\omega_1) \in F_{n,2n-s}$ and $\phi(
\overline{\omega}_1) \in F_{n,2n-s}$.

Remark that if $\omega_1$ is of the form $\ldots \alpha_i^{(0)}\ldots
\alpha_j^{(0)} \ldots \alpha_i^{(1)} \ldots \alpha_j^{(1)} \ldots$,
then $\overline{\omega}_1$ is of the form $\ldots \alpha_{-j}^{(0)}\ldots
\alpha_{-i}^{(0)} \ldots \alpha_{-j}^{(1)} \ldots \alpha_{-i}^{(1)}
\ldots$. It means that in $\phi(\omega_1)$, $j$ is a sub-descendant of
$i$ and in $\phi(\overline{\omega}_1)$, $-i$ is a sub-descendant of
$-j$. Thus $\phi( \overline{\omega}_1)= \overline{\phi(\omega_1)}$.

Moreover
a shuffle between
$\alpha_{j_{s-n+1}}^{(1)} \alpha_{j_{s-n+2}}^{(1)} \ldots\alpha_{j_n}^{(1)}$
and  $\alpha_{-j_n}^{(0)} \ldots
\alpha_{-j_{s-n+2}}^{(0)}\alpha_{-j_{s-n+1}}^{(0)}$
corresponds by $\phi$ to a shuffle between the special leaves of
$\phi(\omega_1)$, $j_{s-n+1}, j_{s-n+2} \ldots, j_{n}$, and the nodes
of $\phi(\overline{\omega}_1)$,
$-j_{n}, \ldots, -j_{s-n+2} , -j_{s-n+1}$.

Since $\omega_1\bowtie \overline{\omega}_1 = \omega_0 \alpha_{j_n}^{(0)} \left[\alpha_{j_{s-n+1}}^{(1)} \alpha_{j_{s-n+2}}^{(1)} \ldots\alpha_{j_n}^{(1)}
\bowtie \alpha_{-j_n}^{(0)} \ldots
\alpha_{-j_{s-n+2}}^{(0)}\alpha_{-j_{s-n+1}}^{(0)} \right ] \alpha_{-j_n}^{(1)} 
\overline{\omega}_0 $, it is straightforward to conclude that 
$\phi(
\omega_1\bowtie\overline{\omega}_1)=\phi(\omega_1) \bowtie
\overline{\phi(\omega_1)}$.
\end{proof}

\section{The number of regions of the type C Catalan arrangement}
\label{sec:enumeration}
We are now able to compute the number of regions of the type $C$ Catalan
arrangement. We first compute the number of rooted ordered
forests of size $n$ with $s$ special leaves.

\begin{proposition}
\label{prop:numla}
   The number of rooted ordered forests of size $n$ with $s$ special
   leaves, $C_{n,s}$, verifies the following formula :
    $ C_{n,s} = \frac{s \binom{2n-s}{n}}{2n-s}$ for $ 1 \leq s \leq
    n.$

\end{proposition}

\begin{proof}
  Every rooted ordered forest $F$ can be identified with an annotated
  1-sketch $\psi(F)$ by Proposition \ref{prop:symbij}.  Let $F$ be a rooted ordered
  forest of size $n$ with $s$ special leaves
  $e_{n-s+1},..., e_{n-1}, e_n$, then let 
  $\psi(F) =\omega_0 \alpha_{e_n}^{(0)}\alpha_{e_{n-s+1}}^{(1)}
  \alpha_{e_{n-s+2}}^{(1)}... \alpha_{e_{n-1}}^{(1)}
  \alpha_{e_n}^{(1)}$. We associate an up step $U$ to each
  $\alpha_{.}^{(0)}$ letter and a down step $D$ to each
  $\alpha_{.}^{(1)}$ letter and thus obviously obtain a Dyck path of
  size $n$.
  This implies that $C_{n,s}$ is equal to the number of
  Dyck paths of size $n$ that have forms $P = P_0 UDD...D$, here each
  $P_0U$ is a lattice path with $n$ up steps and $n-s$
  down steps. Consider the family $L$ of all lattice paths from
  $(0,0)$ to $(2n-s,s)$ consisting of $n$ $U$ and $n-s$ $D$. $L$
  is enumerated by $\binom{2n-s}{n}$.

  Now consider the action of the cyclic group $\Z_{2n-s}$ on $L$ by
  cyclic rotation. Pick a lattice path on $L$, by cycle lemma, there
  exist exactly $s = n-(n-s)$ cyclic rotations that are
  $1$-dominating (any prefix has strictly more $U$ than $D$). Therefore, each orbit contains exactly $s$ lattice
  paths such that any prefix has strictly more $U$ than $D$. Thus
  these $s$ lattice paths are of the forms $UUv$ with $n$ $U$ and
  $n-s$ $D$, and they end at height $s$. Let $L_0$ be the set of all
  lattice paths on $L$ such that any prefix has strictly more $U$ than
  $D$, so the number of elements of $L_0$ is
  $\frac{s \binom{2n-s}{n}}{2n-s}$.

  Given the following bijective transformation for any lattice
  path $P_1$ on $L_0$ by changing $P_1 = UUv$ to $P'_1= UvU = P_0 U,$
  we get that
  $P_0U$ is a lattice path with $n$ up steps and $n-s$ down steps and
  always above the $x$-axis. Now add $s$ down steps at the end of
  $P_0U$, then we get a Dyck path that has form $P = P_0 UDD...D$.

  Therefore, we conclude the formula of $C_{n,s}$ above.
  



\end{proof}



Now we are able to enumerate the regions of the type $C$ Catalan arrangement.
\begin{proposition}
  The number of regions of the type $C$ Catalan arrangement is
  $$ r(C_{\left\{-1,0,1\right\}}(n)) = {2^n}n! \binom{2n}{n} $$
\end{proposition}

\begin{proof}
The number of labeling of a rooted ordered forest of size $n$ with labels
  $e_1, \ldots, e_n$ such that
  $\{|e_1|, \ldots, |e_n|\}=\llbracket 1,n\rrbracket$ is $2^n n!$.
  Thus, from Corollary \ref{cor:regions2symforests}, we can compute the number of
  regions of the type $C$ Catalan arrangement,
  $$ r(C_{\left\{-1,0,1\right\}}(n)) = {2^n}n! \sum\limits_{s=1}^{n} C_{n,s} D_{n,s} $$
  where $C_{n,s}$ is given by Proposition \ref{prop:numla} and
  $D_{n,s}$ is the number of shuffles between any rooted labeled ordered forest
  of size $n$ with $s$ special leaves and its symmetric.

 \quad Now we compute $D_{n,s}$. By Propositions  \ref{prop:symbij} and \ref{prop:compatible}, this is equal to the number of elements of the set 
 $\omega_1\bowtie \overline{\omega}_1 = \omega_0 \alpha_{j_n}^{(0)} \left[\psi
\bowtie \overline{\psi} \right ] \alpha_{-j_n}^{(1)} 
\overline{\omega}_0 $, with $\omega_1 \in A_{n,2n-s}$, $\omega_1= \omega_0
 \alpha_{j_n}^{(0)} \alpha_{j_{n-s+1}}^{(1)}
 \alpha_{j_{n-s+2}}^{(1)}... \alpha_{j_{n-1}}^{(1)}
 \alpha_{j_n}^{(1)}$ and $\psi= \alpha_{j_{n-s+1}}^{(1)} \alpha_{j_{n-s+2}}^{(1)}... \alpha_{j_{n-1}}^{(1)} \alpha_{j_n}^{(1)}$. 
 On the other hand, every annotated 1-sketch of size $n$ can be
 represented by a Dyck path of the same size, so $D_{n,s}$ is
 obviously the number of shuffles between  $s$ down steps and $s$ up
 steps, here each element of the shuffles is of the form $a_1 a_2...a_s b_s...b_2b_1$ with $(a_i, b_i) \in \left\{ (U,D), (D,U) \right\}$, for all $1 \leq i \leq s$. Therefore, $D_{n,s} = 2^{s}$. 
Note that we have the recurrence formula $C_{n,s} = C_{n-1,s-1} + C_{n,s+1} $ for all $ 1 \leq s \leq n-1$. Then we get
$$ \sum\limits_{s=1}^{n} C_{n,s} D_{n,s}=  \sum\limits_{s=1}^{n} \frac{s2^s \binom{2n-s}{n}}{2n-s} =  \binom{2n}{n}$$
by induction.

\end{proof}

It would now be interesting to see if the bijection between our forests
and the regions of the Catalan arrangement of type $C$ can be refined
to the regions of the Linial arrangement of type $C$, thus giving a
bijective interpretation of the enumeration exhibited by
C.A. Athanasiadis in \cite{A96}.

\bibliographystyle{plain}
\bibliography{typeCArrangement}

\end{document}